\documentclass[12pt]{amsart}

\usepackage{mathrsfs}
\usepackage{amssymb}
\usepackage{latexsym}

\usepackage[mathscr]{eucal}

\usepackage[all]{xy}
\usepackage{pb-diagram,pb-xy}
\usepackage{pstricks}

\usepackage[a4paper, hmargin=3.0cm, tmargin=3.2cm, bmargin=3.2cm]{geometry}
\numberwithin{equation}{subsection}

\theoremstyle{plain}
\newtheorem{thm}[equation]{Theorem}

\newtheorem{prop}[equation]{Proposition}
\newtheorem{cor}[equation]{Corollary}

\theoremstyle{definition}
\newtheorem{defn}[equation]{Definition}
\newtheorem{quest}[equation]{Question}

\newcommand{\bC}{\mathbb{C}}
\newcommand{\bD}{\mathbb{D}}
\newcommand{\bF}{\mathbb{F}}
\newcommand{\bN}{\mathbb{N}}
\newcommand{\bP}{\mathbb{P}}
\newcommand{\bQ}{\mathbb{Q}}
\newcommand{\bR}{\mathbb{R}}
\newcommand{\bS}{\mathbb{S}}

\newcommand{\bZ}{\mathbb{Z}}

\newcommand{\gE}{\mathbf{E}}

\newcommand{\cA}{\mathcal{A}}
\newcommand{\cL}{\mathcal{L}}
\newcommand{\cP}{\mathcal{P}}

\newcommand{\actson}{\curvearrowright}

\newcommand{\ch}{\operatorname{ch}}
\newcommand{\Cob}{\operatorname{Cob}}

\newcommand{\cotanh}{\operatorname{coth}}
\newcommand{\BSO}{\operatorname{BSO}}
\newcommand{\dash}{^{\prime}}

\newcommand{\diag}{\operatorname{diag}}
\newcommand{\Diff}{\operatorname{Diff}}
\newcommand{\Eig}{\operatorname{Eig}}
\newcommand{\eins}{\operatorname{\mathbf{1}}}

\newcommand{\eul}{\operatorname{Eul}}
\newcommand{\Fr}{\operatorname{Fr}}
\newcommand{\Fred}{\operatorname{Fred}}
\newcommand{\Gl}{\operatorname{GL}}

\newcommand{\hq}{/\hspace{-1.2mm}/}

\newcommand{\ind}{\operatorname{ind}}
\newcommand{\id}{\operatorname{id}}

\newcommand{\kerv}{\operatorname{Kerv}}
\newcommand{\loopinf}{\Omega^{\infty}}
\newcommand{\MNSO}{\operatorname{MNSO}}
\newcommand{\MTSO}{\operatorname{MTSO}}
\newcommand{\MSO}{\operatorname{MSO}}
\newcommand{\Or}{\operatorname{Or}}
\newcommand{\path}{\operatorname{Path}}
\newcommand{\proj}{\operatorname{proj}}
\newcommand{\PT}{\operatorname{PT}}
\newcommand{\sign}{\operatorname{sign}}
\newcommand{\suspinf}{\Sigma^{\infty}}
\newcommand{\Spec}{\operatorname{Spec}}

\newcommand{\symb}{\operatorname{smb}}
\newcommand{\thom}{\operatorname{th}}
\newcommand{\td}{\operatorname{td}}
\newcommand{\Td}{\operatorname{Td}}
\newcommand{\Th}{\operatorname{Th}}
\newcommand{\bTh}{\mathbb{T}\mathbf{h}}
\newcommand{\trf}{\operatorname{trf}}

\address{Mathematisches Institut der Universit\"at Bonn,
Endenicher Allee 60, 53115 Bonn, Bundesrepublik Deutschland}

\email{ebert@math.uni-bonn.de}

\subjclass{55R40, 58J20, 58J20, 57R90}

\begin{document}
\vspace*{-1cm}

\title[Vanishing theorem]{A vanishing theorem for
characteristic classes of odd-dimensional manifold bundles}

\author{Johannes Ebert}

\begin{abstract}
We show how the Atiyah-Singer family index theorem for both, usual
and self-adjoint elliptic operators fits naturally into the
framework of the Madsen-Tillmann-Weiss spectra. Our main theorem
concerns bundles of odd-dimensional manifolds. Using completely
functional-analytic methods, we show that for any smooth proper
oriented fibre bundle $E \to X$ with odd-dimensional fibres, the
family index $\ind (B) \in K^1 (X)$ of the odd signature operator is
trivial. The Atiyah-Singer theorem allows us to draw a topological
conclusion: the generalized Madsen-Tillmann-Weiss map $\alpha: B
\Diff^+ (M^{2m-1}) \to \loopinf \MTSO(2m-1)$ kills the Hirzebruch
$\cL$-class in rational cohomology. If $m=2$, this means that
$\alpha$ induces the zero map in rational cohomology. In particular,
the three-dimensional analogue of the Madsen-Weiss theorem is wrong.
For $3$-manifolds $M$, we also prove the triviality of $\alpha: B
\Diff^+ (M) \to \MTSO (3)$ in mod $p$ cohomology in many cases. We
show an appropriate version of these results for manifold bundles
with boundary.
\end{abstract}

\maketitle
\setcounter{tocdepth}{1}
\tableofcontents

\section{Introduction and statement of
results}\label{introductionsection}

One of the greatest achievements of algebraic topology in the last
decade are the two proofs of Mumford's conjecture on the homology of
the stable mapping class group by Madsen and Weiss \cite{MW} and by
Galatius, Madsen, Tillmann and Weiss \cite{GMTW}. The
Pontrjagin-Thom construction is crucial for both proofs; it provides
a map from the classifying space of the diffeomorphism group of a
compact surface to the infinite loop space $\loopinf \MTSO (2)$ of
the Madsen-Tillmann-Weiss spectrum, in other words the Thom spectrum of
the inverse of the universal complex line bundle.

The proof in \cite{GMTW} consists of two parts. One part
(essentially due to Tillmann \cite{Till}), exclusively applies to
$2$-dimensional manifolds, because it relies on two deep results of
surface theory (the Harer-Ivanov homological stability theorem and
the Earle-Eells theorem on the contractibility of the components of
the diffeomorphism group of surfaces of negative Euler number). The
other part of the proof, however, is valid for manifolds of
arbitrary dimension and with general ''tangential structures'' and
provides a vast generalization of the classical Pontrjagin-Thom
theorem relating bordism theory of smooth manifolds and stable
homotopy.

Given an oriented (we will ignore more general tangential structures
throughout the present paper) closed manifold $M$ of dimension $n$,
there exists a map

\begin{equation}\label{mtmap}
\alpha_{E_M}: B \Diff^+ (M) \to \loopinf \MTSO(n),
\end{equation}

where $\MTSO(n)$ denotes the Thom spectrum of the inverse of the
universal $n$-dimensional oriented vector bundle. Let $\Cob_{n}^{+}$
be the oriented $n$-dimensional cobordism category: objects are
closed $(n-1)$-dimensional manifolds, morphisms are oriented
cobordisms and composition is given by gluing cobordisms. With a
suitable topology on object and morphism spaces, $\Cob_{n}^{+}$
becomes a topological category. The maps $\alpha$ from \ref{mtmap}
assemble to a map

\begin{equation}\label{gmtwequivalence}
\alpha^{GMTW}: \Omega B \Cob_{n}^{+} \to \loopinf \MTSO(n),
\end{equation}

and the main result of \cite{GMTW} states that $\alpha^{GMTW}$ is a
homotopy equivalence. Moreover, for any closed $n$-manifold $M$,
there is a tautological map $\Phi_M: B \Diff^+ (M) \to \Omega B
\Cob_{n}^{+}$ and $\alpha^{GMTW} \circ \Phi_M = \alpha_{E_M}$.

The exclusive result for two-dimensional manifolds is that when $M$
is a closed connected oriented surface of genus $g$, then $\Phi_M$
induces an isomorphism on integral homology groups of degrees $*
\leq g/2-1$. Both theorems together provide an isomorphism of the
homology of $B \Diff^+ (M)$ and $\loopinf \MTSO(2)$ (in that range
of degrees).

In this paper, we study the map $\alpha_{E_M}$ (or, equivalently,
$\Phi_M$) when $M$ is an oriented closed manifold of \emph{odd}
dimension. It turns out that $\alpha_{E_M}$ fails to be an
isomorphism in homology in any range and that no clue about the
homology of $B \Diff^+ (M)$ can be derived from the study of
$\alpha_{E_M}$. This seems to be an unsatisfactory state of affairs
and therefore we attempt to arouse the reader's curiosity by the
following remark:

Even if the map $\alpha$ fails to be an ''equivalence'' of some
kind, it still contains interesting information about $B
\Diff^+(M)$. Any cohomology class of $\loopinf \MTSO(n)$ (in an
arbitrary generalized cohomology theory) yields, via $\alpha_{E_M}$,
a cohomology class of $B \Diff^+ (M)$, also known as a
characteristic class of smooth oriented $M$-bundles. One should
think of these characteristic classes as ''universal'' classes in
the sense that they are defined for \emph{all} oriented
$n$-manifolds and are defined using only the \emph{local} structure
of the manifold.

Examples are the \emph{generalized MMM-classes} (this is the
abbreviation of the names Mumford, Miller, Morita)

\[
f_! (c (T_v E)) \in H^{*-n} (B;R),
\]

where $f: E \to B$ is a smooth oriented fibre bundle with vertical
tangent bundle $T_v E$, $R$ is a ring and $c \in H^*(BSO(n); R)$ is
a characteristic class of oriented vector bundles. The generalized
MMM-classes come from spectrum cohomology classes of $\MTSO(n)$.

Other examples come from index theory of elliptic operators. Any
sufficiently natural elliptic differential operator on oriented
$n$-manifolds defines a characteristic class in $K^0$ (namely, the
family index). Likewise, a natural self-adjoint elliptic operator
has a family index in $K^{-1}$ and so it defines a characteristic
class in $K^{-1}$. An application of the Atiyah-Singer Index theorem
shows that these index-theoretic classes also come from $\MTSO(n)$.

On any closed oriented Riemannian manifold of odd
dimension, there is the \emph{odd signature operator} $D: \cA^{ev}
(M) \to \cA^{ev} (M)$ on forms of even degree. It is self-adjoint,
elliptic and its kernel is the space of harmonic form of even
degree, which is isomorphic to $H^{ev}(M; \bC)$. Given any smooth
oriented $M$-bundle $f:E \to B$ we can choose a Riemannian metric on
the fibres and study the induced family of elliptic self-adjoint
operators. Here is the central result of the present paper.

\begin{thm}\label{mainresult}
The family index of the odd signature operator on an oriented bundle
$E \to B$ with odd-dimensional fibres is trivial, $\ind (D)=0 \in
K^1 (B)$.
\end{thm}

The proof of this result is entirely \emph{analytic}; it is based on
the fact that the kernel dimension of $D$ is constant. Therefore the
Atiyah-Singer index theorem allows us to draw topological
conclusions from Theorem \ref{mainresult}. Here is one of them:

\begin{thm}\label{vanishing}
For any closed oriented $2m-1$-dimensional manifold $M$, the
Madsen-Tillmann-Weiss map $\suspinf (B \Diff^+(M))_+ \to  \MTSO(2m-1)$
kills the Hirzebruch $\cL$-class $\thom_{-L_{2m-1}} \cL \in
H^{4*-2m+1} (\MTSO (2m-1); \bQ)$.

In particular, for any oriented smooth fibre bundle $f: E \to B$
with fibre $M$, the generalized MMM-class $f_{!} (\cL (T_v E)) \in
H^* (B; \bQ)$ is trivial.
\end{thm}

The precise meaning of this theorem will be clarified in the main
text. If $m=2$, Theorem \ref{vanishing} implies:

\begin{cor}\label{threedim}
If $\dim M= 3$, the Madsen-Tillmann-Weiss map $\alpha: B \Diff^+(M) \to
\loopinf \MTSO(3)$ is trivial in rational cohomology (in positive
degrees).
\end{cor}

This is an amusing result. Recently, Hatcher and Wahl \cite{HW}
showed an analogue of the Harer-Ivanov homological stability for
mapping class groups of $3$-manifolds. Moreover, for large classes
$3$-dimensional manifolds, it is known that the components of the
diffeomorphism group are contractible (but that tends to become
wrong after stabilization). One might be tempted to think that these
results helps to make the proof of the analogue of the Mumford
conjecture valid, leading to a description of the stable homology of
mapping class groups of $3$-manifolds in terms of the homology of
$\loopinf \MTSO(3)$. Corollary \ref{threedim} shows that this is not
the case.

Here is another consequence of Theorem \ref{vanishing}:

\begin{cor}\label{Lueckranicki}
Let $E \to B$ be an oriented fibre bundle over a closed oriented
manifold with odd-dimensional closed fibres. Then $\sign (E) =0$.
\end{cor}

This is an old theorem, which was first mentioned without proof by
Atiyah \cite{Atfib} (perhaps the proof Atiyah had in mind is along
the lines of the argument of the present paper). Proofs of
\ref{Lueckranicki} were given by Meyer \cite{Mey} and L\"uck/Ranicki
\cite{LR}. In fact, \ref{vanishing} and \ref{Lueckranicki} are
equivalent, as we will see in subsection \ref{applications}.

There is a version of Theorem \ref{vanishing} for manifold bundles
with boundary, such that the boundary is trivialized (section
\ref{boundarysection}).

In dimensions of the form $4r+1$, there is a real refinement of
Theorem \ref{mainresult}. More precisely, the odd signature operator
has an index in real K-theory. This real index, however, is usually
not zero. This is discussed in section \ref{realsection}

Theorem \ref{mainresult} is stronger than \ref{vanishing}, because
it also has consequences in mod $p$-cohomology. We prove two things
for oriented $3$-manifolds in that direction. Fix an oriented
$3$-manifold $M$. We will prove (in section \ref{vanishingtorsion}):

\begin{itemize}
\item Fix $k \geq 1$. Then for almost all odd primes $p$, the map
$\alpha^*: H^{4k-1} (\MTSO(3); \bF_p) \to H^{4k-1} (B \Diff^+ (M);
\bF_p)$ is zero (Theorem \ref{adamsmethod}).
\item Fix an odd prime $p$. Then $\alpha^*: H^{4k-1} (\MTSO(3); \bF_p) \to H^{4k-1} (B \Diff^+ (M);
\bF_p)$ is zero for an infinite number of values for $k$ (Theorem
\ref{wumethod}).
\end{itemize}

In both cases, the primes to which the theorem applies does not
depend on $M$.

In a companion paper \cite{Eb} we show that all cohomology classes
in $H^{* > 0} (\MTSO(2m); \bQ)$ are detected on some bundle of
$2m$-manifolds and that all classes in $H^{* > 0} (\MTSO(2m+1);
\bQ)$ which are not multiples of the Hirzebruch $\cL$-class are
detected on some $2m+1$-dimensional bundle. This means that Theorem
\ref{vanishing} is the only vanishing theorem of this type.

\subsection{Outline of the paper}

Section \ref{backgroundmtspectra} is a survey on the stable homotopy
theory which is needed in this paper. We briefly discuss general
Thom spectra, the Madsen-Tillmann-Weiss spectra, the Pontrjagin-Thom
construction, the Madsen-Tillmann-Weiss map and Thom isomorphisms.
Subsection \ref{homotopyproperties} is devoted to a study of the
component group $\pi_0 (\MTSO(n))$. This is needed later in section
\ref{realsection}. Section \ref{backgrounindex} provides the
necessary constructions from index theory. In section
\ref{oddsigop}, we discuss the odd signature operator and prove
Theorem \ref{mainresult}. Also, we show \ref{vanishing} and
\ref{threedim}. Section \ref{realsection} discusses the real index
of the odd signature operator. Finally, in section
\ref{vanishingtorsion}, we discuss the vanishing theorem in finite
characteristic. Section \ref{boundarysection} discusses the
extension of the results to the bounded case.

\subsection{Acknowledgements}

The author is indebted to a number of people for enlightening
discussions about the mathematics in this paper. Among them are
Oscar Randal-Williams, Ulrike Tillmann, Ib Madsen and Ulrich Bunke.
Last but not least, I have to acknowledge the hospitality of the
Mathematical Institute of the University of Oxford and the financial
support from the Postdoctoral program of the German Academic
Exchange Service (DAAD) which I enjoyed when this project was begun.

\section{Background material on Madsen-Tillmann-Weiss
spectra}\label{backgroundmtspectra}

In this section, we review some material on the
Madsen-Tillmann-Weiss spectra. Most of the results are standard
except perhaps those concerning the component group of $\MTSO (n)$
in \ref{homotopyproperties}. For the subsections \ref{thomspectra},
\ref{orientation} and \ref{pontthomsection}, the reader is referred
to the textbook \cite{Rudyak} for proofs and much more details.

\subsection{Stable vector bundles and their Thom
spectra}\label{thomspectra}

For our purposes, a \emph{stable vector bundle} $V$ on a space $X$
is a map $\xi_V:X \to \bZ \times BO$. The \emph{rank} of $V$ is the
locally constant function $X \stackrel{\xi_V}{\to} \bZ \times BO \to
\bZ$. Given two ordinary vector bundles $V_i \to X$ of rank $r_i$,
$i=0,1$, we can form their formal difference by the following
procedure. Let $\mu: \bZ \times BO \times \bZ \times  BO \to \bZ
\times  BO$ be the Whitney sum map and $\iota: \bZ \times  BO \to
\bZ \times  BO$ the inversion map. Furthermore, let $\xi_i: X \to
\bZ \times BO$ be classifying maps for $V_i$ (composed with the
inclusion $BO(r_i) \to BO$), then $V_0 - V_1$ is the stable vector
bundle which is given by the composition

\[
X \stackrel{\xi_0,\iota \circ \xi_1}{\to} \bZ \times BO \times \bZ
\times BO \stackrel{\mu}{\to} \bZ \times BO.
\]

The rank of $V_0 - V_1$ is $r_0-r_1$. Clearly, we can add and
subtract stable vector bundles by means of the maps $\mu$ and
$\iota$. Furthermore, any ordinary vector bundle can be considered
as a stable vector bundle.

The \emph{Thom space} of a vector bundle $V \to X$ is the space $\Th
(V) = X^V = \bD (V) / \bS (V)$, the quotient of the unit disc bundle
by the unit sphere bundle. The \emph{Thom spectrum} $\bTh(W)$ of a
stable vector bundle $W$ of rank $d$ is produced as follows. Let
$X_n := \xi_{W}^{-1}(\{d \} \times BO_{d+n})$; these subspaces form
an exhaustive filtration $X_{-d} \subset X_{1-d} \subset \cdots
\subset X$.  Let $W_n := \xi_{W}^{*} L_{d+n}$ be the pullback of the
$d+n$-dimensional universal vector bundle. Clearly, there is an
isomorphism $W_{n+1} |_{X_n} \cong \bR \oplus W_n$. The $n^{th}$
space of $\bTh(W)$ is the Thom space $X_n^{W_n}:= \bD (W_n) / \bS
(W_n)$ of $W_n$ and the structure maps are
\[
\Sigma X_n^{W_n} \cong X_n^{\bR \oplus W_n } \cong
X_n^{W_{n+1}|_{X_n}} \hookrightarrow X_{n+1}^{W_{n+1}}.
\]
The homotopy type of the spectrum $\bTh(W)$ depends only on the
homotopy class of $\xi_{W}$. Furthermore, if $W$ is an ordinary
vector bundle, then the Thom spectrum is homotopy equivalent to the
suspension spectrum $\suspinf X^{W}$ of the Thom space of $W$. In
particular, the Thom spectrum of the trivial $0$-dimensional bundle
$\underline{0}$ on $X$ is $\suspinf X_+$. Let $W$ be a stable vector
bundle and $V$ an ordinary vector bundle. Given the description
above, it is not hard to see that there is an inclusion map

\[
\bTh (W) \to \bTh (W \oplus V).
\]

Let $V \to X; W \to Y$ be two stable vector bundles. There is a
canonical homotopy equivalence $\bTh (V) \wedge \bTh (W) \simeq \bTh
(V \times W)$. If $X=Y$, we get a diagonal map $\diag:\bTh (V \oplus
W) \to \bTh (V) \wedge \bTh (W)$. A special case is the diagonal
$\bTh (V) \to \suspinf X_+ \wedge \bTh (V)$.

\subsection{Orientations and Thom isomorphisms}\label{orientation}

Assume that $A$ is an associative and commutative ring spectrum with
unit (the rather old-fashioned notion of \cite{Adams} is sufficient
for our purposes). Let $V \to X$ be a stable vector bundle of rank
$d \in \bZ$. The cohomology $A^* (\bTh(V))$ is a graded left $A^*
(X)$-module; a pair $(x,y) \in A^n (X) \times A^m (\bTh (V))$ is
sent to the composition
\[
x \cdot y:\bTh (V) \stackrel{\diag}{\to}  \suspinf X_+ \wedge \bTh
(V) \stackrel{x \wedge y}{\to} \Sigma^n A \wedge \Sigma^m A {\to}
\Sigma^{n+m} A.
\]
A \emph{Thom class} or \emph{$A$-orientation} of $V$ with
$A$-coefficients is a cohomology class $v \in A^d (\bTh (V))$ such
that for any $x \in X$, the image of $v$ under the restriction map
$A^d (\bTh (V)) \to A^d (\bTh (V_x)) \cong A^d (\bS^d) \cong A^0
(*)$ is a unit. This is equivalent to saying that $A^* (\bTh (V))$
is a free $A^*(X)$-module on the generator $v$ or that he map
$\thom_{V}^{A}:A^* (X) \to A^{*+d} (\bTh (V))$; $x \mapsto x \cdot
v$ is an isomorphism. If this is the case, then $\thom_{V}^{A}$ is
called the \emph{Thom isomorphism}. If $A$ is understood, then the
superscript is often omitted.

More generally, we can define a \emph{relative} Thom isomorphism.
Let $V$ be a stable vector bundle of rank $d$ and let $W$ be another
stable vector bundle of rank $e$. Assume that $V$ has a Thom class
$v$. Let $\thom_{W,W \oplus V}^{A}:A^* (\bTh (W)) \to A^{*+d} (\bTh
(W \oplus V))$ be the homomorphism which maps $x \in A^{n}(\bTh
(W))$ to the composition

\[
\bTh (W \oplus V) \stackrel{\diag}{\to} \bTh (W) \wedge \bTh (V)
\stackrel{x \wedge v}{\to} \Sigma^{n} A \wedge \Sigma^{d} A \to
\Sigma^{n+d} A;
\]
this is an isomorphism of $A^{*}(X)$-modules. If $v \in
A^{d}(\bTh(V))$ and $w \in A^{e} (\bTh (W))$ are Thom classes, then
$\thom_{W, W \oplus V}^{A}(v)$ is a Thom class for $V \oplus W$. If
the Thom classes of different stable vector bundles are chosen
compatibly in this way, then the Thom isomorphisms are compatible in
the sense that $\thom_{U \oplus V, U \oplus V \oplus W} \circ
\thom_{U, U \oplus V} = \thom_{U, U \oplus V \oplus W}$ whenever $U$
is an arbitrary stable vector bundle. We shall use the short
notation $\thom_{V}: A^*(\bTh (W)) \to A^*(\bTh (V \oplus W))$ if
$W$ is understood.

\textbf{Examples:} The examples of ring spectra which play a role in
this paper are Eilenberg-Mac Lane spectra $H R$ for commutative
rings $R$ as well as the complex $K$-theory spectrum $K$. It is
well-known that a vector bundle which is oriented in the ordinary
sense has a preferred\footnote{depending on the choice of a
generator of $H_1 (\bR; \bR\setminus 0; \bZ)$.} $H \bZ$-Thom class
and so it has a $H R$-Thom class for any ring $R$. A stable vector
bundle has an $H \bZ$-orientation if and only if $w_1 (V)=0$. Any
complex vector bundle has a $K$-orientation and so does every
complex stable vector bundle, i.e. a formal difference of complex
vector bundles. However, there are several choices for these
$K$-orientation. We follow the convention that the Thom class of a
complex vector bundle $\pi:V\to X$ of rank $n$ is represented by the
complex

\[
0 \to \pi^* \Lambda^0 V \stackrel{v \wedge}{\to} \pi^* \Lambda^1 V
\stackrel{v \wedge}{\to} \pi^* \Lambda^2 V \stackrel{}{\to} \ldots
\pi^* \Lambda^n V \to 0.
\]

The following observation is important for index theory. Let $V \to
X$ be a real vector bundle. Then $V \otimes \bC$ has a natural
$K$-orientation. Therefore there is a relative Thom isomorphism

\begin{equation}\label{relthomiso}
K^* (\bTh (-V)) \cong K^* (\bTh (- V \oplus V \otimes \bC)) \cong
K^* (\bTh (V)).
\end{equation}

In this equation we used Bott periodicity to identify $K^*$ with
$K^{*+2}$. We will do this throughout the whole paper.

\subsection{The Pontrjagin-Thom construction}\label{pontthomsection}

Let $M$ be a closed smooth oriented manifold of dimension $n$ and
let $\Diff^+ (M)$ be the group of diffeomorphisms of $M$ endowed
with the Whitney $C^{\infty}$-topology. We will study \emph{smooth
oriented $M$-bundles}, i.e. fibre bundles $f: E \to B$ with
structural group $\Diff^+ (M)$ and fibre $M$. Let $Q \to B$ be the
associated $\Diff(M)$-principal bundle. The \emph{vertical tangent
bundle} is the oriented vector bundle $T_v E :=Q \times_{\Diff^+
(M)} TM \to Q \times_{\Diff^+ (M)} M = E$. The \emph{normal bundle}
of $f$ is the stable vector bundle $\nu(f):= - T_v E$.

If $B$ is paracompact, then there is a fat embedding $j: E \to B
\times \bR^{\infty}$, i.e. $\proj \circ j =f$ and the image of $j$
has a tubular neighborhood $U$. Moreover, the space of such fat
embeddings is contractible. Collapsing everything outside $U$ to the
basepoint defines a map of spectra

\[
\PT_f:\suspinf B_+ \to \bTh (\nu(f)),
\]

the \emph{Pontrjagin-Thom map} (or PT-map, for short). The map
$\PT_f$ depends on a contractible space of choices. In particular,
its homotopy class only depends on $f$. For more details on the
PT-construction in this parameterized setting, see \cite{GMT},
section 3.

The Pontrjagin-Thom map can be used to define the \emph{umkehr map}
in generalized cohomology. Let $f: E \to B$ be a smooth fibre bundle
of dimension $n$, $A$ a ring spectrum and we assume that $\nu(f)$
has an $A$-orientation. The umkehr homomorphism $f_{!}:A^*(E) \to
A^{*-n}(B)$ is defined as the composition

\begin{equation}\label{defumkehr}
A^* (E) \stackrel{\thom_{\nu(f)}}{\to} A^{*-n} (\bTh (\nu(f)))
\stackrel{\PT_{f}^{*}}{\to} A^{*-n} (B).
\end{equation}

The original application of the Pontrjagin-Thom construction was to
give a bordism theoretic description of the homotopy groups of Thom
spectra (or vice versa). Here is the most general version of this
correspondence.

\begin{thm}\label{pontthomtheorem}
Let $V \to X$ be a stable vector bundle of rank $-n \in \bZ$. If $-n
> 0$, then $\pi_0 (\bTh (V))=0$. If $n \geq 0$, then the group
$\pi_0 (\bTh (V))$ is isomorphic to the bordism group of triples
$(M^n,g,\phi)$, where $M^n$ is a closed smooth manifold, $g: M \to
X$ a continuous map and $\phi: \nu (M) \cong g^* V$ a stable vector
bundle isomorphism. Two triples $(M_0,g_0,\phi_0)$ and
$(M_1,g_1,\phi_1)$ are bordant if there exists a bordism $N$ from
$M_0$ to $M_1$, a continuous map $h:N \to X$ such that $h|_{M_i} =
g_i$ and a stable bundle isomorphism $\psi:\nu (N) \oplus \bR\cong
h^* V$ whose restriction to $M_i$ is the isomorphism $\nu (N) \oplus
\bR \cong \nu (M_i) \stackrel{\phi_i}{\cong} g_{i}^{*} V$.

Given a triple $(M,g,\phi)$, $c:M \to *$ the constant map, then the
corresponding element in $\pi_0 (\bTh (V))$ is the composition

\[
\suspinf \bS^0 \stackrel{\PT_c}{\to} \bTh (\nu (M)) \stackrel{g,
\Phi}{\to} \bTh (V).
\]
\end{thm}

A detailed proof of this well-known result can be found in
\cite{Rudyak}, ch. IV \S 7. Of course, this also gives an
interpretation of the groups $\pi_k (\bTh (V)) \cong \pi_{0} (\bTh
(V - \bR^k))$.

\subsection{Madsen-Tillmann-Weiss spectra and Madsen-Weiss maps}

Let $n \geq 0$, let $BSO(n)$ be the classifying space for oriented
Riemannian $n$-dimensional vector bundles and let $L_n \to BSO(n)$
be the universal oriented vector bundle. The reader should note that
the space $BSO(0)$ is homotopy equivalent to the two-point space
$\bS^0$ and therefore it is \emph{not} the classifying space for the
group $SO(0)$. The most natural explanation for this phenomenon
occurs in the framework of stacks. Let $\Or (\bR^n)$ be the set of
orientations of the vector space $\bR^n$; the group $O(n)$ acts on
$\Or (\bR^n)$. The stack of oriented $n$-dimensional vector bundles
is the quotient stack $\Or (\bR^n) \hq O(n)$. For $n \geq 1$, the
$O(n)$-action on $\Or (\bR^n)$ is transitive and hence $\Or (\bR^n)
\hq O(n) \cong \ast \hq SO(n)$, while for $n=0$, we have $\Or (\bR^0
)\hq O(0) \cong \bS^0$.

\begin{defn} The Thom spectrum of the stable
vector bundle $- L_n$ on $BS
O(n)$ is called the
\emph{Madsen-Tillmann-Weiss spectrum} (or MTW-spectrum) and it is
denoted by $\MTSO(n)$. Moreover, we denote by $\MNSO (n)$ the Thom
spectrum of $L_n$.
\end{defn}

Let $f:E \to B$ be a smooth oriented $M$-bundle. Recall that the
space of orientation-preserving bundle maps $\lambda:T_v E \to L_n$
is contractible. Therefore the orientation defines a contractible
space of maps $\kappa_E=\bTh (\lambda):\bTh (-T_v E) \to \MTSO(n)$.
The \emph{Madsen-Tillmann-Weiss map} (or MTW-map) of the bundle $f:E
\to B$ is the composition

\begin{equation}\label{defmtmap}
\alpha_E:=\kappa_E \circ \PT_f: \suspinf B_+ \to \MTSO(n),
\end{equation}

which is defined uniquely up to a contractible space of choices.

For the universal oriented $M$-bundle $E_M \to B \Diff^+ (M)$, we
obtain a universal MTW-map

\[
\alpha_{E_M}: \suspinf (B \Diff^+ (M))_+ \to \MTSO(n).
\]

On the other extreme, the constant map $M \to \ast$ is a smooth
oriented $M$-bundle and its MTW-map is a map $\alpha_M: \suspinf
\bS^0 \to \MTSO (n)$.

By the Thom isomorphism,

\[
H^*(BSO(n); R)=H^*(\suspinf BSO(n)_+;R) \cong H^{*-n} (\MTSO(n);R).
\]

The cohomology of $BSO(n)$ is well-known. For example, if $\bF$ is a
field of characteristic different from $2$, then

\begin{equation}\label{cohobso}
H^*(BSO(2m+1); \bF) \cong \bF[p_1, p_2, \ldots p_m]; \; \;H^*
(BSO(2m);\bF) \cong \bF[p_1, \ldots p_m, \chi]/ (\chi^2 -p_m).
\end{equation}

Let $f: E \to B$ be an oriented $n$-dimensional manifold bundle, let
$\alpha_E: \suspinf B_+ \to  \MTSO(n)$ be its MTW-map. An element $c
\in H^* (BSO(n))$ can be interpreted as a characteristic class for
oriented $n$-dimensional vector bundles and therefore we write
$c(T_v E) \in  H^* (E)$ for the pullback $\bar{\lambda}^* c$, where
$\bar{\lambda}: E \to BSO(n)$ is any map underlying a bundle map
$T_v E \to L_n$.

\begin{prop}\label{beckergottlieb}
Let the notations be as above. Then

\[\alpha^{*}_{E}  \thom_{- L_n} (c)=f_{!} (c (T_v (E))) \in H^{*-n} (B).\]
\end{prop}

\begin{proof}
By definition

\[
\alpha^*  \thom_{-  L_n} (c) \stackrel{\ref{defmtmap}}{=} \PT^{*}_{f} \bTh (\lambda)^{*} \thom_{- L_n} (c) = \\
\PT_{f}^{*} \thom_{-T_v E} ( c (T_v E)).\]

The second equality expresses the compatibility of Thom isomorphisms
and pullbacks.
\end{proof}

Therefore any $c \in H^* (BSO(n))$ defines a characteristic class of
oriented $n$-manifold bundles. We call these classes ''generalized
MMM-classes'', because the case $n=2$, $c= \chi^{i+1}$ gives the
classes $\kappa_i$ defined by \emph{M}umford \cite{Mum},
\emph{M}iller \cite{Miller} and \emph{M}orita \cite{Mor}.

Recall the adjunction between the two functors $\Sigma^{\infty}$ and
$\loopinf$: given a spectrum $\gE$ and a space $X$, there is a
natural bijection

\[
[X, \loopinf \gE] \cong [\suspinf X_+, \gE].
\]

Under this adjunction, $\alpha_E$ corresponds to a map $B \to
\loopinf \MTSO (n)$, which is the original MTW-map studied in
\cite{MT}, \cite{MW}, \cite{GMTW}. We will call this adjoint by the
same name and denote it by the same symbol. There is no danger of
confusion, because we keep our notation for spaces and spectra
entirely disjoint. For the more computational purposes of the
present paper, the spectra point of view is more transparent and
convenient.

The adjoint $\suspinf (\loopinf \gE)_+ \to \gE$ of the identity on
$\loopinf \gE$ induces a map

\[
s: A^* (\gE) \to A^* (\loopinf \gE),
\]

the \emph{cohomology suspension}, whenever $A$ is a spectrum. If $A
= H \bQ$, then the right-hand-side is a graded-commutative
$\bQ$-algebra, but the left-hand-side is only a graded $\bQ$-vector
space. Let $\Lambda$ denote the functor which associates to a graded
module the free, graded-commutative algebra it generates; $s$
extends to an algebra homomorphism

\begin{equation}\label{milmooreiso}
s: \Lambda (H^{*>0} (\gE; \bQ)) \to H^* (\loopinf_0 \gE; \bQ).
\end{equation}

This is an isomorphism by a classical result of algebraic topology,
see \cite{MilMoo}, p. 262 f.

If $X$ is a space and $\suspinf X_+ \to \gE$ a map with adjoint $X
\to \loopinf \gE$, then the following diagram commutes

\begin{equation}\label{cohosusp}
\xymatrix{
A^* (\gE) \ar[d]^{s}\ar[r] & A^{*} (\suspinf X_+) \ar@{=}[d]\\
A^* (\loopinf \gE) \ar[r] & A^* (X). }
\end{equation}

The rational cohomology of $\loopinf \MTSO(n)$ can be easily
computed using \ref{cohobso}, \ref{cohosusp} and \ref{milmooreiso}.

\subsection{The component group of Madsen-Tillmann-Weiss
spectra}\label{homotopyproperties}

There are several maps which relate the spectra $\MTSO(n)$ for
different values of $n$.

\begin{itemize}
\item The obvious bundle isomorphism $L_{n+1} |_{BSO(n)} \cong L_{n}
\oplus \bR$ induces a map of spectra $\eta: \MTSO (n) \to \Sigma
\MTSO(n+1)$.
\item The inclusion $-L_{n+1} \to \underline{0}$ of stable
vector bundles on $BSO(n+1)$ yields a spectrum map $\omega: \MTSO
(n+1) \to \suspinf BSO(n+1)_+$.
\item The MTW-map of the oriented $\bS^{n}$-bundle $BSO(n) \to BSO(n+1)$
is a map $\beta: \suspinf BSO(n+1)_+ \to \MTSO(n)$.
\end{itemize}

\begin{prop}\label{cofibseq}
The maps $\eta$, $\omega$ and $\beta$ form a cofibration sequence

\begin{equation}\label{fibersequence}
\MTSO (n+1) \stackrel{\omega}{\to} \suspinf BSO(n+1)_+
\stackrel{\beta}{\to} \MTSO(n) \stackrel{\eta}{\to} \Sigma \MTSO
(n+1).
\end{equation}
\end{prop}

\begin{proof}
This follows immediately from Lemma 2.1 in \cite{Galspin}.
\end{proof}

The (homotopy) colimit of the sequence

\[\MTSO (0) \stackrel{\eta}{\to} \Sigma \MTSO(1) \stackrel{\eta}{\to}
\Sigma^2 \MTSO(2) \to \ldots\]

is the universal Thom spectrum $\widetilde{\MSO}$, the Thom spectrum
of the universal $0$-dimensional stable vector bundle $-L \to \BSO$
(which becomes $\bR^n -L_n$ when restricted to $\BSO(n)$). The usual
universal Thom spectrum $\MSO$ is the Thom spectrum of $ L \to
\BSO$. The spectra $\widetilde{\MSO}$ and $\MSO$ are homotopy
equivalent: Let $ \iota: \BSO \to \BSO$ be the inversion map, such
that $\iota^* L = -L$. The map $\iota$ is covered by a bundle map
$j:-L \to L$ which induces a homotopy equivalence $\bTh (j):
\widetilde{\MSO} \to \MSO$

The long exact homotopy sequence induced by \ref{fibersequence}
shows that the map $\eta_*: \pi_i (\MTSO(n)) \to \pi_i( \Sigma
\MTSO(n+1))$ is an epimorphism if $i \leq 0$ and an isomorphism if
$i < 0$. Therefore the inclusion $\Sigma^n \MTSO(n) \to
\widetilde{\MSO}$ yields an isomorphism $\pi_{i} (\MTSO(n)) \cong
\pi_{n+i} (\widetilde{\MSO}) \cong \pi_{n+i} (\MSO) \cong
\Omega^{SO}_{n+i}$ (the oriented bordism group) for all $ i <0$.
Therefore from \ref{cofibseq}, we get a commutative diagram; the rows are exact and the vertical maps are isomorphisms:

\begin{equation}\label{folge}
\xymatrix{
\pi_0 (\MTSO(n+1)) \ar[r] \ar[d] & \pi_0 (\suspinf \BSO(n+1)_+) \ar[r] \ar[d] & \pi_0
(\MTSO(n)) \ar[r] \ar[d] &   \pi_{-1} (\MTSO(n+1)) \ar[d] \ar[r] & 0\\
\pi_0 (\MTSO(n+1)) \ar[r]  & \bZ \ar[r] & \pi_0
(\MTSO(n)) \ar[r] &  \Omega_{n}^{SO}  \ar[r] & 0.
}
\end{equation}

In order to study these groups further, we  use the
bordism-theoretic interpretation of $\pi_0 (\MTSO(n))$ provided by
Theorem \ref{pontthomtheorem}.

It can be rephrased in such a way that $\pi_0(\MTSO(n))$ is the
bordism group of oriented $n$-manifolds, where $M_0$ and $M_1$ are
considered to be bordant if and only if there exists an oriented
bordism $N$ between them, an oriented $n$-dimensional vector bundle
$V$ on $N$ and a stable bundle isomorphism $TN \cong V \oplus \bR$.
Clearly we can assume that $N$ has no closed component and
therefore, there is an actual isomorphism of vector bundles $TN
\cong V \oplus \bR$ by elementary obstruction theory. In other
words, there is a nowhere vanishing tangential vector field on $N$
which is the inward normal vector field on $M_0$ and the outward
normal vector field on $M_1$. This bordism group is also known as
Reinhardt's bordism group, see \cite{Rein}.

The maps in \ref{folge} have the following interpretation:

\begin{enumerate}
\item $\pi_0(\MTSO(n+1)) \to \bZ$ sends the bordism class of an
oriented $n+1$-manifold $M$ to its Euler number $\chi(M)$ (which is
a cobordism invariant in Reinhardt's bordism).
\item $\bZ \to \pi_0 (\MTSO (n))$ sends $1$ to the bordism class of $\bS^n$.
\item $\pi_0(\MTSO (n)) \to \Omega_{n}^{SO}$ is the forgetful map.
\end{enumerate}

Only the first claim needs a further justification. If $f:E \to B$ is an oriented
$n+1$-manifold bundle, then the composition $\suspinf B_+
\stackrel{\alpha_E}{\to} \MTSO(n+1) \stackrel{\omega}{\to} \suspinf
\BSO(n+1)_+$ is the composition of the Becker-Gottlieb transfer
$\suspinf B_+ \to \suspinf E_+$ (see \cite{BG}) with the classifying
map $\suspinf E_+ \to \suspinf \BSO(n)_+$ of $T_v E$. Therefore, if
$B$ is connected (and $n+1>0$), then the induced map on $\bZ = \pi_0
\suspinf B_+ \to \suspinf BSO(n+1)_+ = \bZ$ is the multiplication by
the Euler number $\chi(M)$ of $M$, by Theorem 2.4 of \cite{BG}.

Let $\eul_n \subset \bZ$ be the subgroup generated by all Euler
numbers of oriented $n$-manifolds; the exact sequence \ref{folge}
induces

\begin{equation}\label{pinullsequence}
0 \to \bZ / \eul_{n+1} \to \pi_0 (\MTSO(n)) \to \Omega_{n}^{SO} \to
0.
\end{equation}

The group $\eul_n \subset \bZ$ is easily computed. Its values are

\[
\eul_n =
\begin{cases}
0;   &  n  \not \equiv 0 \pmod 2;\\
2 \bZ ;& n \equiv 2 \pmod 4;\\
 \bZ ; & n \equiv 0 \pmod 4.
\end{cases}
\]

The first case is clear. The third case follows from $\chi
(\bS^{4k}) =2$ and $\chi (\bC \bP^{2k}) = 2k+1$. The second case is
implied by $\chi(\bS^{4k+2})=2$ and the congruence $\chi(M^{4k +2})
\equiv 0 \pmod 2$ which follows from Poincar\'e duality in a
straightforward manner.

The sequence \ref{pinullsequence} is always split, as we show now.
If $M$ is an oriented closed $(4m+1)$-dimensional manifold, then the
\emph{Kervaire semi-characteristic} $\kerv (M) \in \bZ/2$ is defined
to be $\kerv (M) := \sum_{i \geq 0} b_{2i} (M) = \sum_{i=0}^{2m}
\dim b_{i}(M)$, where $b_i$ is the real Betti number of $M$. By the
proposition below, $\kerv (M)$ defines a homomorphism $\pi_0
(\MTSO(4m+1)) \to \bZ/2$.

\begin{prop}
The Kervaire semi-characteristic of an oriented $4m+1$-manifold $M$
only depends on its bordism class in $\pi_0 (\MTSO(4m+1))$.
\end{prop}

\begin{proof}
It is enough to show the following: If $N^{4m+2}$ is a connected
oriented manifold with boundary $M$ and if there is a nowhere
vanishing vector field on $N$ which is normal to the boundary, then
$\kerv (M)=0$. Clearly, the double $d N$ of $N$ is closed and has a
vector field without zeroes; thus $\chi(d N)=0$ and therefore
$\chi(N)=0$.
Let $A$ be the image of $H^{2m+1}(M,N) \to H^{2m+1}(N)$.

Look at the long exact sequence of the pair $(N,M)$ in real
cohomology:

\[
 0 \to H^0 (N,M)  \to H^0 (N) \to H^0 (M) \to
 \ldots \to H^{2m} (M) \to H^{2m+1} (N,M) \to A \to 0.
 \]

We compute (in $\bZ/2$)

\begin{eqnarray*}
0 = \sum_{i=0 }^{2m+1}b_i (N;M) + \sum_{i= 0}^{2m}b_i (N) + \sum_{i= 0}^{2m}b_i (M) + \dim A =& \\
= \sum_{i=2m+1}^{4m+2} b_i (N)   + \sum_{i= 0}^{2m}b_i (N) + \sum_{i= 0}^{2m}b_i (M) + \dim A  =& \\
= \chi(N) + \kerv (M) + \dim A = \kerv (M) + \dim A.
\end{eqnarray*}

By Poincar\'e duality, the cup product pairing on $A$ is skew-symmetric and nondegenerate, thus $\dim A \equiv 0 \pmod 2$.
\end{proof}

Let us summarize the description of the component group of
$\MTSO(n)$.

\begin{enumerate}
\item If $n \equiv 3 \pmod 4$, then $\pi_0 (\MTSO(n)) \cong
\Omega_{n}^{SO}$
\item If $n \equiv 2 \pmod 4$, then the sequence \ref{pinullsequence}
splits by $\pi_0 (\MTSO(n)) \to \bZ$; $[M] \mapsto \frac{1}{2}
\chi(M)$.
\item If $n \equiv 0 \pmod 4$, then \ref{pinullsequence}
is also split. If $M^{4m}$ is an oriented manifold, then $\sign (M)
+ \chi(M) \equiv 0 \pmod 2$ is immediate from the definition of the
signature and the Euler number and from Poincar\'e duality. The map
$\pi_0 (\MTSO(n)) \to \bZ$, $[M] \mapsto \frac{1}{2}(\sign (M) +
\chi(M))$ is a splitting.
\item If $n \equiv 1 \pmod 4$, then \ref{pinullsequence} is split by the
Kervaire semi-characteristic $\pi_0 (\MTSO(n)) \to \bZ/2$.
\end{enumerate}

\section{Background material on index theory}\label{backgrounindex}

In this section, we will present background material on index theory
for bundles of compact manifolds. For details, the reader is
referred to either the original source \cite{ASIV} or to the
textbook \cite{LM}.

There are two types of the $K$-theoretic index theorem: One for
usual elliptic operators and another one for self-adjoint elliptic
operators on a fibre bundle $f: E \to B$. In the former case, the
index is an element in $K^0(B)$ while in the latter one we get an
index in\footnote{We are using the Bott periodicity theorem without
mentioning it. Therefore we identify $K^1$ with $K^{-1}$.} $K^1
(B)$.

Assume that $f:E \to B$ is a smooth fibre bundle on a paracompact
space $B$ with compact closed fibres. Assume that a fibrewise smooth
Riemannian metric on the vertical tangent bundle $T_v E$ is chosen.
All vector bundles on $E$ will be fibrewise smooth (i.e. the
transition functions are smooth in the fibre-direction) and all
hermitian metrics on vector bundles are understood to be smooth. All
differential operations, like exterior derivatives and connections,
will be fibrewise.

For an hermitian vector bundle $V \to E$, we denote $\Gamma_B (V) =
\bigcup_{x \in B} \Gamma (E_x ; V_x) $, where $E_x = f^{-1}(x)$ and
$V_x = V|_{E_x}$. This family of vector spaces over $B$ can be made
into a vector bundle (of Fr\'echet spaces) by requiring that a
section $s:B \to \Gamma_B (V)$ is continuous if the associated
section of $V \to E$ is continuous in the $C^{\infty}$-topology.
Using the metrics on $T_v E$ and $V$ and a connection on $V$, we can
define the $L^2$-Sobolev norms $\| \ldots \|_{r}$ on $\Gamma_B (V)$,
for all $r \geq 0$. The completion with respect to this norm is a
Hilbert bundle which we denote by $W^{2,r}_{B}(V)$.

There is a technical problem to overcome at this point; it is
discussed and solved in \cite{Atiseg}, pp. 5, 13 f., 38-43. Namely,
it is not quite true that the structural group of $W^{2,r}_{B} (V)$
is general linear group of an infinite-dimensional Hilbert space.
The reason is that the action $\Diff (M) \actson W^{2,s} (M)$ is
continuous only in the sense that $\Diff(M) \times W^{2,s} (M) \to
W^{2,s} (M)$ is continuous, but not $\Diff(M) \to \Gl( W^{2,s})$
when the latter has the norm topology. Instead, this map is
continuous when $\Gl ( W^{2,s})$ has the compactly generated
compact-open topology. Denote by $\Gl ( W^{2,s})_{co}$ the group
with this topology. Then $\Gl ( W^{2,s})_{co}$ is contractible (this
is much easier than Kuiper's theorem which asserts that $\Gl (
W^{2,s})$ is contractible). Moreover, $\Gl ( W^{2,s})_{co}$ acts
continuously by conjugation on the space of Fredholm operators with
a suitably redefined topology. This new space of Fredholm operators
is homotopy equivalent to the original one.

Therefore the Hilbert bundles $W^{2,r}_{B}(V)$ are trivial and the
trivialization is unique up to homotopy (in fact, the space of
trivializations is contractible).

Let $V_0, V_1 \to E$ be two hermitian vector bundles and let $D: V_0
\to V_1$ be a vertical elliptic operator of order $m$. Then $D$ has
an extension to the bundle of Sobolev spaces $D: W^{2,s+m}_{B} (V)
\to W^{2,s}_{B}(V)$, which consist of Fredholm operators. We choose,
for any vector bundle $V$, an elliptic pseudodifferential operator
$A$ of order $-m/2$ which is invertible (for example
$A_V=(1+\nabla^* \nabla)^{-m/4}$ will do for any connection $\nabla$
on $V$). The operator $A_{V_1} D A_{V_0}$ has order $0$ and so it
induces a family of Fredholm operator $W^{2,0}_{B}(V_0) \to
W^{2,0}_{B}(V_1)$. After an application of the trivializations
above, we get a continuous map, denoted $\ind (D)$:

\[
\ind (D) : B \to \Fred (H),
\]

where $H$ is a fixed separable, infinite-dimensional Hilbert space.
The Atiyah-J\"anich theorem states that the $\Fred (H)$ is a
classifying space for complex $K$-theory and therefore we get an
element $\ind (D) \in K^0 (B)$. It does not depend on the choices
involves.

On the other hand, if $D:\Gamma_B (V) \to \Gamma_B (V)$  is a
formally self-adjoint elliptic operator of order $m  \geq 0$, we get
an index in $K^1 (B)$. Here we consider the operator $A_{V} D
A_{V}^{*}$, which is elliptic of order $0$ and formally
self-adjoint. It has the same kernel and the same positive and
negative spectral spaces as the original $D$.

Thus we get a self-adjoint bounded Fredholm operator $D:
W^{2,0}_{B}(V) \to W^{2,0}_{B}(V)$. In the same way as for ordinary
elliptic operators, we get a map $B \to \Fred_{s.a.}(H)$, where
$\Fred_{s.a.}^{}(H) $ is the space of self-adjoint Fredholm
operators on $H$ endowed with the norm topology. Let
$\Fred^{\pm}_{s.a.}(H) \subset \Fred_{s.a.}(H)$ be the subspace
consisting operators $A$ such that $\pm A$ is essentially positive
(an operator is \emph{essentially positive} if there exists an
$A$-invariant subspace $U \subset H$ of finite codimension, such
that $A|_{U}$ is positive definite). The two spaces
$\Fred_{s.a}^{\pm}(H) \subset \Fred_{s.a.}(H)$ are open, closed and
contractible (in fact, $\Fred_{s.a.}^{\pm}(H)$ is star-shaped with
center $\pm \id$).

Let $\Fred_{s.a.}^{0}= \Fred_{s.a.}(H) \setminus
(\Fred_{s.a.}^{+}(H) \cup \Fred_{s.a.}^{-}(H))$. Atiyah and Singer
\cite{AS69} showed that it has a very interesting topology: it has
the homotopy type of the infinite unitary group $U(\infty)$. Thus it
is a representing space for $K^{-1}$.

Returning to the self-adjoint family of operators $D$ on $E \to B$,
the map $B \to \Fred_{s.a.}(H)$ defines an element $\ind (D) \in K^1
(B)$ (if $D$ is essentially definite, this element is trivial).

\subsection{The topological index}

Let $f: E \to B$ be a smooth proper bundle, $\pi: T:=T^{*}_{v} E \to
E$ the vertical cotangent bundle and $\pi_0:\bS (T^{*}_{v} E) \to E
$ its unit sphere bundle. Let $D: \Gamma_{B} (V_0) \to \Gamma_{B}
(V_1)$ an elliptic differential operator. Recall that the
\emph{symbol} of $D$ is a bundle map $\symb_D:\pi^* V_0 \to \pi^*
V_1$ which is an isomorphism outside the zero section (this is the
definition of ellipticity). If $D$ has order $1$, then the symbol is
\[
\symb_D ( \xi) v = i( D (f s) - f Ds),
\]
where $\xi $ is a vertical cotangent vector at $x \in E$, $f$ is a
smooth function such that $df_x = \xi$ and $s$ is a section of $V_0$
such that $s(x)=v$. For higher orders, there is a more complicated
formula, which we will not need here.

We will constantly identify the vertical cotangent and the vertical
tangent bundle. The symbol $\symb_D$ defines the \emph{symbol class}
$[\symb_D]_0 \in K^0 (T; T \setminus 0) = K^0 (E^{T})$ of $D$.

Following \cite{APSIII}, we can associate a symbol class
$[\symb_D]_{1} \in K^{-1}(E^{T})$ to a self-adjoint elliptic
operator $D$. Consider the symbol $\symb_D: \pi^* V \to \pi^* V$. It
is a self-adjoint endomorphism of $\pi^* V$ and it is an isomorphism
away from the zero section. Let $\tilde{\pi}: T \oplus \bR \to E$.
We define $[\symb_D]_{1}$ to be the class in $K^{-1}(E^{T}) =
K^{0}((T,T \setminus 0) \times (\bR , \bR \setminus 0))$ represented
by the complex

\[
0 \to \tilde{\pi}^{*} V \stackrel{\tilde{\symb_D}}{\to}
\tilde{\pi}^{*} V \to 0,
\]

where $\tilde{\symb_D}$ is given at the point $(x,t) \in T \oplus
\bR$ by $\tilde{\symb_D}_{(x,t)} := (\symb_D)_x - i t \eins$.
Actually, \cite{APSIII} give a different formula, but the passage
between the two formulations is by an elementary deformation. We
leave it to the reader to figure that out.

Recall the relative Thom isomorphism \ref{relthomiso} $\thom_{-T_v E
\otimes \bC}: K^* (E^{T_v E}) \to K^{*} (E^{-T_v E})$. The
Atiyah-Singer family index theorem (\cite{ASIV} for the usual case,
\cite{APSIII} for the self-adjoint case) states that in both cases
($i=0,1$)

\begin{equation}\label{indexthm}
\ind (D) = \beta^{-d}\PT^{*}_{f} \thom_{-T_v E \otimes \bC}
([\symb_D]_i) \in K^{i}(B).
\end{equation}

\subsection{Universal operators}

Now we assume that the first order elliptic operator $D$ on the
$n$-dimensional oriented bundle $E \to B$ family is \emph{universal
on the symbolic level}. By that expression, we mean that there exist
$SO(n)$-representations $W_0$ and $W_1$ and an $SO(n-1)$-equivariant
isomorphism $\gamma:W_0 \to W_1$ such that

\begin{enumerate}
\item as Hermitian vector bundles, $V_0$ and $V_1$ are isomorphic to the
associated bundle $\Fr_{v} (E) \times_{SO(n) } W_i \to E$;
\item the symbol $\symb_D$ restricted to the unit cotangent sphere
bundle equals the bundle map $\Fr_{v} (E) \times_{SO(n-1)} W_0 \to
\Fr_{v} (E) \times_{SO(n-1)} W_1$ induced by $\gamma$.
\end{enumerate}

The trivial vector bundles $ \bR^n \times W_i$ on $\bR^n$ are
$SO(n)$-equivariant and the map $\gamma$ defines an
$SO(n)$-equivariant isomorphism $ \bS^{n-1} \times W_0 \to \bS^{n-1}
\times W_0$. Therefore, $(W_0,W_1, \gamma)$ defines a class
$\sigma_D \in K^{0}_{SO(n)} (\bD^n, \bS^{n-1})$. The image of
$\sigma_D$ under the standard homomorphism $K^{0}_{SO(n)}(\bD^n,
\bS^{n-1}) \to K^{0} (\bD (L_n), \bS(L_n)) \cong K^0 (\MNSO(n))$ is
denoted by the same symbol. Clearly, the class $\sigma_D$ pulls back
to the symbol class $[\symb_D]$ under the map $\bTh (T_v E) \to
\MNSO(n)$. By the Thom isomorphism $K^0 (\MNSO(n)) \cong K^0 (\MTSO
(n))$, we get a class $\thom \sigma_D \in K^0 (\MTSO(n))$.

The index theorem for the symbolically universal operator $D$ on the
fibre bundle $f: E \to B$ reads:

\begin{equation}\label{univindex}
\ind (D) = \alpha_{E}^{*} \thom \sigma_D.
\end{equation}

Similarly, if $V_0 = V_1 = V$ and $\gamma$ is self-adjoint, we get a
class $\sigma_D \in K^1 (\MNSO(n))$ and $\thom \sigma \in K^1
(\MTSO(n))$ and the index theorem is expressed by the same formula
as in \ref{univindex}.

\section{The index theorem for the odd signature
operator}\label{oddsigop}

\subsection{The signature operators}

Let $M$ be a closed oriented Riemannian manifold of dimension $n$.
Recall that there is the Hodge star operator\footnote{There might
exist different sign conventions about $\ast$. We are constantly
using the definition given in \cite{ASIII}.} $\ast: \cA^{k}(M) \to
\cA^{n-k} (M)$. The star operator is an complex-linear isometry and
satisfies $\ast \ast = (-1)^{k(n-k)}: \cA^{k}(M) \to \cA^k(M)$. The
adjoint $d^{ad}: \cA^{k}(M) \to \cA^{k-1}(M)$ of the exterior
derivative can be written as $d^{ad} = (-1)^{n(k+1)+1} \ast d \ast$.

If $n = 2m$, then one introduces the involution $\tau:= i^{k (k-1) +
m} \ast$ on $k$-forms \cite{ASIII}, p. 574. Then $D_{2m} = d +
d^{ad}$ satisfies $D_{2m} \tau = - \tau D_{2m}$. If
$\cA^{*}_{\pm}(M)$ denote the $\pm 1$-eigenbundles of $\tau$, then
the operator $D: \cA^{*}_{+} (M) \to \cA^{*}_{-} (M)$ is the (even)
\emph{signature operator}. This is an elliptic differential operator
of order $1$ whose index is the same as the signature of $M$.

Following \cite{APSIII}, we introduce the odd signature operator on
a $2m-1$-dimensional closed oriented Riemannian manifold $M$. Note
that $\ast \ast = 1$ and $d^{ad}= (-1)^{k} \ast d \ast : \cA^k (M)
\to \cA^{k-1}(M)$. The \emph{odd signature operator} $D=D_{2m-1}:
\bigoplus_{p \geq 0} \cA^{2p}(M) \to \bigoplus_{p \geq 0}
\cA^{2p}(M)$ is defined to be
\[
D_{2m-1} \phi = i^{m} (-1)^{p+1} (\ast d - d \ast ) \phi
\]
whenever $\phi \in \cA^{2p}(M)$.

A straightforward, but tedious, calculation shows that
\begin{enumerate}
\item $D$ is formally self-adjoint and elliptic.
\item $D ^2 = \Delta = (d + d^{ad})^2$, the Laplace-Beltrami operator.
\end{enumerate}
Moreover, one observes that
\begin{equation}
\ker (D) = \ker (\Delta) = \bigoplus_{p \geq 0} H^{2p} (M; \bC)
\end{equation}
and that consequently
\begin{equation}\label{fundamental}
\dim \ker D = \sum_{p \geq 0} \dim H^{2p}(M; \bC),
\end{equation}

which is the main property needed for the proof of Theorem
\ref{mainresult}.

Both signature operators are symbolically universal. The even one is
associated with the $SO(2m-1)$-equivariant isomorphism of
$SO(2m)$-representations

\[
i ( \epsilon - \ast \epsilon \ast):\Lambda^{*}_{+} (\bR^{2m})\otimes
\bC \to \Lambda^{*}_{-} (\bR^{2m}) \otimes \bC,
\]

where $\epsilon$ denotes wedge multiplication with the last standard
basis vector.

The odd signature operator is associated with the representation
$\Lambda^{ev}(\bR^{2m-1}) $ and the endomorphism

\begin{equation}\label{univsymbsign}
i^{m-1}(-1)^{p} (\ast \epsilon  - \epsilon \ast).
\end{equation}

Abbreviate $\bR^{n}_{0} := \bR^n \setminus 0$. Let $\sigma_{2m-1}
\in K^{-1}_{SO(2m-1)} (\bR^{2m-1}, \bR^{2m-1}_{0})$ be the universal
symbol class of the odd signature operator and $\sigma_{2m} \in
K^{0}_{SO(2m)} (\bR^{2m},\bR^{2m}_{0})$ be the universal symbol
class of the even signature operator. We denote their images in $K^*
(\MNSO(n))$ by the same symbol.

\subsection{The vanishing theorem}

Let $f: E \to B$ be a smooth oriented $M$-bundle, $M$ a closed
oriented $(2m-1)$-manifold. Assume that we choose a Riemannian
metric on the vertical tangent bundle (for any bundle on a
paracompact base space such a metric exists; the space of these
metrics is contractible). The odd signature operators on the fibres
of $f$ fit together to a family of self-adjoint elliptic
differential operators. Therefore we have the family index
\[
\ind (D) \in K^1 (B),
\]
which does not depend on the auxiliary Riemannian metric, but which
is an invariant of smooth oriented $M$-bundles. In the universal
case, we get an element $\ind (D) \in K^1 (B \Diff^+(M))$.

The proof of Theorem \ref{mainresult} is an immediate consequence of
\ref{fundamental} and Theorem \ref{spectralgap} below. Theorem
\ref{spectralgap} is well-known to some people working in operator
theory, see e.g. \cite{Bunke}, 5.1.4. and it is certainly implicitly
contained in \cite{AS69}. I have included the following rather
elementary proof for the convenience of the reader.

\begin{thm}\label{spectralgap}
Let $B$ be a space and let $A: B\to \Fred_{s.a}^{0}(H)$, $x \mapsto A_x$ be a
continuous map such that $x \mapsto \dim \ker A_x$ is locally
constant. Then $A$ is homotopic to a constant map.
\end{thm}

\begin{proof}

\textit{Step 1:} First we show that we can deform $A$ into a family
$A \dash$ consisting of invertible operators.

To this end, we note that because the dimension of $\ker A_x$ is
locally constant, the union $\ker (A):=\bigcup_{x  \in B} \ker
(A_x)$ is a (finite-dimensional) vector bundle on $B$. Therefore the
projection operator $p_x$ onto the kernel of $A_x$ depends
continuously on $x$ and $p_x$ commutes with $A_x$ because $A_x$ is
self-adjoint. Therefore $A_x + t p_x$ is Fredholm for all $t \in
\bR$ and $\Spec (A_x + t p_x) = \Spec A_x \setminus \{0\} \cup \{ t
\} \subset \bR_{\neq 0}$. Thus for $t \neq 0$, $A_x + t p_x$ is
invertible (and neither essentially negative nor positive).

\textit{Step 2:} By step 1, we assume that $A_x$ is invertible for
all $x \in B$, in other words $\Spec (A_x) \subset \bR \setminus 0$
for all $x \in B$. Let $h: \bR \setminus 0 \to \bR$ be the signum
function. For any $x$ and any $t \in [0,1]$, the operator
\[
t h(A) + (1-t)A
\]
is a self-adjoint invertible operator (the latter statement is easy
to see because $A$ and $h(A)$ commute). For $t=0$, we get $A$ and
for $t=1$, we get $h(A)$ which is a self-adjoint involution which is
neither essentially positive nor negative.

\textit{Step 3:} By step 2, we can assume that $A$ is a map from $B$
into the space $\cP(H)$ of all involutions $F$ on $H$ such that
$\Eig (F, \pm 1)$ are both infinite-dimensional. Let us show that
$\cP(H)$ is contractible. The unitary group $U(H)$ acts transitively
on $\cP(H)$ (by conjugation) and the isotropy group at a given $F_0 $ is
$U(\Eig(F_1, 1)) \times U (\Eig(F_0, -1))$. Thus we have a continuous
bijection
\[
U(H) / U(\Eig(F_0, 1)) \times U (\Eig(F_0, -1)) \to \cP.
\]
The map $U(H) \to \cP(H)$, $u \mapsto u F_0 u^{-1}$ has a local
section\footnote{Here is a construction of the local section. Let
$H_{\pm} := \Eig(F_0; \pm 1)$. For a given $F$, let $u_F$ be
$\frac{1}{2}(1 \pm F)$ on $H_{\pm}$. The operator $u_F$ depends
continuously on $F$; $u_{F_0} = 1$. Therefore, for $F$ close to
$F_0$, $u_F$ is isomorphism. An application of the Gram-Schmidt
process defines a continuous family $F \mapsto u_F$ of unitary
operators on a neighborhood of $F_0$ such that $u_F (H_{\pm}) =
\Eig(F;\pm 1)$, in other words, $ u_F F u_{F}^{-1} = F_0$, which is
what we want.} and thus the bijection above is a homeomorphism. The
left hand side space is contractible by Kuiper's theorem \cite{Kuip}
and the long exact homotopy sequence, which completes the proof of
the theorem.
\end{proof}

In the proof of the theorem we had the choice between two different
contractible spaces of nullhomotopies of $A$; in the first step, we
could choose either a positive value or a negative value of the real
parameter $t$ (put in another way: the spectral value can be pushed
either in the positive or in the negative direction). The
concatenation of these two nullhomotopies defines a map $B \to
\Omega \Fred_{s.a.}^{0} \simeq \Omega U \simeq \bZ \times BU$, in
other words an element in $K^0 (B)$. It is not hard to see that this
is the same as the class of the bundle $\ker (A) \to B$.

In the case of the odd signature operator on the smooth oriented
fibre bundle, this is the $K$-theory class of the flat bundle
$\bigoplus_{p \geq 0} H^{2p}(E/B; \bC)$ of even cohomology groups.
This $K$-theory class is a characteristic class of smooth oriented
fibre bundle, nevertheless, it is \emph{not} induced by an element
in $K^0 (\MTSO(2m-1))$. This can be seen as follows. There exist
odd-dimensional manifold bundles $f: E \to B$ such that $\sum_{p
\geq 0} [H^{2p}(E/B; \bC)] \neq 0 \in K^0 (B)$. For example, ones
takes orientation reversing involutions on $\bS^1$ and $\bS^2$. The
diagonal action $\bZ/2 \actson \bS^1 \times \bS^2$ is then
orientation-preserving. The bundle $E \bZ/2 \times_{\bZ/2} (\bS^1
\times \bS^2)\to B \bZ/2$) has the desired property. On the other
hand $K^{0}_{SO(2m-1)} (\bR^{2m-1}, \bR^{2m-1}_{0}) \cong K^{1 +
\tau}_{SO(2m-1)}(*)$ by the Thom isomorphism in twisted $K$-theory
(\cite{FHT}). Here $\tau$ is the twist induced from the central
extension $Spin^c (2m-1) \to SO(2m-1)$. On the other hand, $K^{1 +
\tau}_{SO(2m-1)} =0$, see \cite{FHT}, p.11. By the Atiyah-Segal
completion theorem \cite{AtSegCom}, it follows that $K^0
(\MTSO(2m-1)) \cong K^0 (\MNSO(2m-1))=(K^{0}_{SO(2m-1)} (\bR^{2m-1},
\bR^{2m-1}_{0}))^{\wedge}=0$.

\subsection{Cohomology calculation}

In this section, we indicate how Theorem \ref{vanishing} is derived
from Theorem \ref{mainresult}. The computation is at least
implicitly done in \cite{ASIII} and \cite{APSIII} and we shall give
only a sketch. First note that the second statement of Theorem
\ref{vanishing} is an immediate consequence of the first one in view
of \ref{beckergottlieb}. The Atiyah-Singer index theorem
\ref{univindex} implies that Theorem \ref{mainresult} is equivalent
to the following result; this formulation is what we actually need
in the sequel.

\begin{thm}\label{symbolgoestonull}
Let $\sigma_{2m-1} \in K^1 (\MNSO(2m-1))$ be the universal symbol
class of the signature operator. Then for any smooth oriented bundle
$E \to B$ of $(2m-1)$-dimensional closed manifolds, we have
$\alpha_{E}^{*} \thom (\sigma_{2m-1})=0$.
\end{thm}

Consider the following commutative diagram

\begin{equation}\label{cohomologydiagram}
\xymatrix{
K^0 (\MNSO(2m)) \ar[d]^{\thom_{-L_{2m} \otimes \bC}} \ar[r]^{\rho^*} & K^0 (\Sigma^{1} \MNSO(2m-1)) \ar[d]^{\thom_{-L_{2m-1} \otimes \bC}} \\
K^0 (\MTSO(2m) \ar[r]^{\eta^*} \ar[d]^{\ch} & K^0 (\Sigma^{-1}
\MTSO(2m-1))
\ar[d]^{\ch}\\
H^* (\MTSO(2m); \bQ) \ar[r]^{\eta^*} & H^* (\MTSO(2m-1); \bQ). }
\end{equation}

Let $\tilde{\cL} \in H^* (BSO; \bQ)$ be the multiplicative sequence
in the Pontrjagin classes associated with the formal power series
$\sqrt{x} \cotanh (\frac{\sqrt{x}}{2})$. Recall that the Hirzebruch
$\cL$-class is associated with with the formal power series
$\sqrt{x} \cotanh (\sqrt{x})$. Note that the degree $4k$ parts in
$H^* (BSO(2m))$ are related by

\begin{equation}\label{power2relation}
\tilde{\cL}_{4k} =2^{m-k} \cL_{4k} \in H^{4k} (BSO(2m); \bQ).
\end{equation}

\begin{prop}\label{atiyahcomputation}
\begin{enumerate}
\item The image of $\sigma_{2m} \in K^{0}(\MTSO(2m))$ under the restriction homomorphism $\rho^*:K^{0}(\MTSO(2m)) \to K^0 (\Sigma
\MTSO(2m-1))=K^1 (\MTSO(2m-1))$ coincides with $2 \sigma_{2m-1}$.
\item The image of $\sigma_{2m}$ under $\ch \circ \thom_{-L_{2m}
\otimes \bC}$ in $H^* (\MTSO(2m); \bQ)$ is the class $\tilde{\cL}$.
\end{enumerate}
\end{prop}

\begin{proof}
The first part is contained in the proof of Lemma 4.2 in
\cite{APSIII}. The second part is done in \cite{ASIII}, section 6.
We leave it to the reader to translate the proofs into the present
more abstract notation.
\end{proof}

Theorem \ref{vanishing} follows immediately from
\ref{symbolgoestonull}, \ref{power2relation},
\ref{atiyahcomputation}.

\subsection{Applications of Theorem \ref{vanishing}}\label{applications}

\begin{proof}[Proof of Corollary \ref{threedim}:] Recall the power series expansion

\[
\sqrt{x} \cotanh (\sqrt{x})=\sum_{k=0}^{\infty}
\frac{2^{2k}B_{2k}}{(2k)!} x^{2k}
\]

and recall that $B_{2k}$ is a nonzero rational number. On the other
hand $H^* (BSO(3))= \bQ[p_1]$ and therefore $\cL=
\sum_{k=0}^{\infty}  \frac{2^{2k}B_{2k}}{(2k)!} p_{1}^{k}$. Thus the
components of $\cL$ form an additive basis of $H^*(BSO(3))$.
Therefore, by Theorem \ref{vanishing}, $H^* (\MTSO(3); \bQ) \to H^*
(\suspinf (B \Diff^+ (M))_+; \bQ)$ is trivial. By \ref{milmooreiso},
this finishes the proof.
\end{proof}

\begin{proof}[Proof of Corollary \ref{Lueckranicki}] Because $TE \cong f^* TB \oplus T_v E$, we have
\begin{equation}\label{computation}
\sign (E) = \langle \cL (T E) ; [E] \rangle = \langle
\cL(f^* TB) \cL (T_v E) ; [E] \rangle = \langle \cL(TB) f_{!}(\cL (T_v E))  ; [B]
\rangle.
\end{equation}

By Theorem \ref{vanishing}, $f_{!}(\cL (T_v E))=0$.
\end{proof}

To derive \ref{vanishing} from \ref{Lueckranicki}, observe first
that $H_{*} (B \Diff^+ (M); \bQ) \cong \Omega_{*}^{fr} (B \Diff^+
(M)) \otimes \bQ$ (the framed bordism group) by Pontrjagin's theorem
and Serre's finiteness theorem. Therefore, to show that
$\alpha_{M}^{*}  \thom \cL =0$, it suffices to show that $h^*
\alpha_{M}^{*}  \thom \cL =0$ whenever $h: B \to B \Diff^+(M)$ is a
map with $B$ a framed manifold, classifying an $M$-bundle $f:E \to
B$. If $B$ is framed, then $\cL(TB)=1$ and therefore by
\ref{computation} and \ref{Lueckranicki}

\begin{equation}
0 = \sign (E) = \langle \cL (T E) ; [E] \rangle  = \langle f_{!}(\cL
(T_v E)) \cL(TB) ; [B] \rangle  = \langle f_{!}(\cL (T_v E)) ; [B]
\rangle.
\end{equation}

Therefore \ref{vanishing} follows.

\section{A real refinement and the one-dimensional
case}\label{realsection}

Let $f: E \to B$ be a smooth oriented fibre bundle of fibre
dimension $2m-1$. Recall the formula for the odd signature operator:
$D =  i^{m} (-1)^{p+1} (\ast d - d \ast ) $ on a
$(2m-1)$-dimensional manifold. If $m$ is odd $(2m-1 = 1,5, \ldots$),
then $-iD$ is a \emph{real, skew-adjoint} operator, acting on
real-valued differential forms. As such, it has an index in $KO^{-1}
(B)$, compare \cite{ASV}.

The question we consider is whether this refined index is also
trivial. We have to check whether the argument in the proof of
Theorem \ref{spectralgap} goes through with $-iD$ instead of $D$ in
the space of real, skew-adjoint Fredholm operators. It turns out
that step 2 can be changed appropriately (we deform an invertible
operator into one with $F^2 = -1$). The argument for step 3 can be
applied to the space of skew-adjoint real Fredholm operators $F$
with $F^2 = -1$, because Kuipers theorem is true for the isometry
group of a real Hilbert space as well. The problem is with step 1.

Let $H \cong \bigoplus_{p \geq 0} H^{2p}(E/B;\bR)\to B$ be the
finite-dimensional real vector bundle formed out of the kernels of
the real odd signature operator.

In order to make sense out of the deformation in step 1, it is not
enough to know that $H$ is a real vector bundle, but also that $H$
admits a skew-adjoint invertible endomorphism. Such an endomorphism
is, up to homotopy, the same as a complex structure on $H$.
Therefore:

\begin{thm}\label{realrefinement}
Let $f:E \to B$ be an oriented smooth $M$-bundle, $M$ of dimension
$4r+1$. Then the real family index of the odd signature operator
$\ind_{\bR} D \in KO^{-1}(B)$ is trivial if and only if the the
$K$-theory class $[H] \in KO^0 (B)$ lies in the image of the
realification map $K^0 (B) \to KO^{0} (B)$.
\end{thm}

The first obstruction to find a complex structure on $H$ is of
course the parity of its dimension $\dim H \pmod 2$. This agrees
with the Kervaire semi-characteristic $\kerv (M)$. More generally,
we can interpret this result in terms of the exact sequence

\[
K^{-2}(B) \stackrel{\gamma}{\to} KO^{0 } (B) \stackrel{\delta}{\to}
KO^{-1}(B),
\]

compare e.g. \cite{Kar}, Thm 5.18. The map $\gamma$ is the inverse
to the Bott map, composed with the realification map $K^{0} \to
KO^{0}$ and $\delta$ is the product with the generator of
$KO^{-1}(*) \cong \bZ/2$. The image $\delta ([H]) \in KO^{-1}(B)$
agrees with the real index. So the real index vanishes if and only
if there is a complex structure on $H$.

It is worth to study the $1$-dimensional case explicitly. The
MTW-spectrum is $\MTSO(1) \cong \Sigma^{-1} \suspinf \bS^0$. It is
well-known that $\Diff^+ (\bS^1) \simeq \bS^1$; therefore
$B\Diff^+(\bS^1) \simeq \bC \bP^{\infty}$. The MTW-map
$\alpha:\suspinf \bC\bP^{\infty}_+ \to \Sigma^{-1} \suspinf \bS^0$
can be identified with the \emph{circle transfer}. The restriction
$\suspinf \bS^0 \to \Sigma^{-1} \suspinf \bS^0$ of $\alpha$ to the
basepoint is simply the generator $\eta \in \pi_1 (\suspinf \bS^0)
\cong \bZ/2$.

The odd signature operator on $\bS^1$ is simply $D =-i \ast d$ on
$C^{\infty} (\bS^1)$. If $\bS^1$ has a Riemannian metric with volume
$a$ and $x$ is a coordinate $\bS^1 \to \bR / a \bZ$ preserving
orientation and length, then $D= -i \frac{d}{dx}$. The symbol is
$\symb_D (dx)  = -1 $. Using this, it is easy to see that $\sigma_1
\in K^{-1}_{SO(1)} (\bR, \bR_0) = K^0 (\bR^2, \bR^{2}_{0})$ is the
Bott class. Thus the universal symbol is a generator of
$K^{-1}(\MTSO(1)) \cong \bZ$.

The vanishing theorem \ref{mainresult} in this case can be obtained
much easier because $K^{-1}(\bC \bP^{\infty})=0$. In fact, the
vanishing theorem for the topological index follows immediately from
this fact, without any use of elliptic operator theory.

On the other hand, the Kervaire semi-characteristic of $\bS^1$ is
clearly nonzero and therefore the real index of the signature
operator is nonzero; it is a generator of $KO^{-1} (\bC
\bP^{\infty}) \cong \bZ/2$ (the latter isomorphism follows easily
from the main result of \cite{And}). The restriction of the real
index to the basepoint is the generator of $KO^{-1}(\ast) = \bZ/2$.

\section{Vanishing theorems in mod $p$ cohomology and an open problem}\label{vanishingtorsion}

We have seen that for any oriented closed $3$-manifold $M$, the map
$\alpha_{E_M}:B \Diff^+ (M) \to \MTSO(3)$ is trivial in rational
cohomology. What we do not know is whether there exists an oriented
closed $3$-manifold $M$ such that $\alpha_{E_M}:B \Diff^+ (M) \to
\loopinf \MTSO(3)$ is nontrivial in homology.

In this section, we sketch two methods to derive from
\ref{mainresult} that $\alpha_{E_M}^{*}:H^{4k-3}(\MTSO(3); \bF_p)
\to H^{4k-3}(B \Diff^+ (M); \bF_p)$ vanishes for certain values of
$k$ and primes $p$.

\begin{thm}\label{adamsmethod}
For any oriented closed $3$-manifold $M$ and for any $k \geq 1$, the
map $\alpha_{E_M}^{*}:H^{4k-3}(\MTSO(3); \bF_p) \to H^{4k-3}(B
\Diff^+ (M); \bF_p)$ is trivial for all primes $p$ with $p \geq 2k$
and $p$ not dividing the numerator of $B_k$ (these are almost all
primes, for a fixed $k$).
\end{thm}

\begin{thm}\label{wumethod}
For any oriented closed $3$-manifold $M$ and for any odd prime $p$,
the map $\alpha_{E_M}^{*}:H^{4k+1}(\MTSO(3); \bF_p) \to H^{4k+1}(B
\Diff^+ (M); \bF_p)$ is trivial when $k=\frac{1}{2}(p-1)i$ for some
$i \in \bN$.
\end{thm}

Note that neither of the sets of pairs $(k,p)$ provided by the two
theorems contains the other one. Also, they do not exhaust all
values of $(k,p)$. Neither theorem makes a statement about the prime
$2$. The methods of the proof of both theorems can be used to derive
vanishing theorems for all odd dimensions (and the method of
\ref{wumethod} gives a result about the prime $2$ as well), but here
we confine ourselves to the case of dimension $3$.

\begin{proof}[Proof of Theorem \ref{adamsmethod}:]
The symbol of the odd signature operator $\thom (\sigma) \in
K^{-1}(\MTSO(3))$, when considered as a map $\MTSO(3) \to
\Sigma^{-1}K$, can be lifted to connective $K$-theory, i.e. to a map

\[
\kappa:\MTSO(3) \to \Sigma^{-3} \mathbf{k}.
\]

The composition $\kappa \circ \alpha:B \Diff^+ (M)_+ \to \MTSO(3)
\to \Sigma^{-3} \mathbf{k}$ is still nullhomotopic. The theorem
follows from a theorem of Adams \cite{AdCh1}, \cite{AdCh2} about the
spectrum cohomology of $\mathbf{k}$. In general, the class $s_r :=
r! \ch_r \in H^{2r}(BU; \bZ)$ is not a spectrum cohomology class,
i.e. it does not lie in the image of the cohomology suspension
$H^{2r}(\mathbf{k}; \bZ) \to H^{2r}(BU; \bZ)$. The result of
\cite{AdCh1}, \cite{AdCh2} is that a certain multiple $m(r)\ch_r$
actually \emph{is} a spectrum cohomology class. The number $m(r) $
is given by $m(r):= \prod_{p} p^{[\frac{r}{p-1}]}$. The product goes
over all prime numbers and for $x \in \bR$, $[x]$ is the largest
integer which is less or equal than $x$ (thus, it involves only
primes $p$ with $p-1 \leq r$). Moreover, $u_r :=m(r) \ch_r$ is a
generator of $H^{2r}(\mathbf{k}; \bZ)\cong \bZ$, $r \geq 0$ (all
other cohomology groups of $\mathbf{k}$ are trivial). If $p$ is an
odd prime then $H^{*} (\MTSO(3); \bZ)$ has no $p$-torsion and so
$H^* (\MTSO(3); \bZ) \otimes \bF_p \cong H^* (\MTSO(3); \bF_p)$.
Therefore, if $\kappa^* (\Sigma^{-3} u_r) \in H^{2r-3} (\MTSO(3))$
reduces to a generator of $H^{2r-3}(\MTSO(3); \bF_p)$, then $\alpha:
H^{2r-3}(\MTSO(3); \bF_p) \to H^{2r-3}(B \Diff^+(M); \bF_p)$ is the
zero map. Up to powers of $2$ which we can disregard since $p$ is
assumed to be odd, $\kappa^*$ maps $\Sigma^{-3} u_{2r} \in H^{2r-3}
(\Sigma^{-3}\mathbf{k})$ to

\[
\pm \frac{B_r}{(2r)!} m(2r) (u_{-3} p_{1}^{r}).
\]

(it is not hard to derive that this class is integral from Von
Staudt's theorem and Lemma 2.1. of \cite{Pap}). This reduces to a
generator mod $p$ if $p$ does not divide $\frac{B_r}{(2r)!} m(2r)$
which is certainly the case if $p \geq 2r$ and $p$ does not divide
the numerator of $B_r$.
\end{proof}

\begin{proof}[Proof of Theorem \ref{wumethod}:]

Look at the diagram:

\begin{equation}
\xymatrix{ H^1(\MTSO(3); \bF_p)  \ar[d] & H^1(\MTSO(3); \bZ)
\ar@{>>}[l] \ar[r] \ar[d] & H^1(\MTSO(3); \bQ) \ar[d]^{0}\\
H^1(B \Diff^+ (M); \bF_p) & H^1(B \Diff^+ (M); \bZ)\ar[l]
\ar@{>->}[r] & H^1(B \Diff^+ (M); \bQ) }
\end{equation}

The right-hand vertical arrow is zero by Corollary \ref{threedim}
and the lower right horizontal arrow is injective ($H^1(X; \bZ) \to
H^1 (X; \bQ)$ is injective for any space $X$). The upper left
horizontal arrow is surjective since $p$ is odd. An easy diagram
chase shows that $H^1 (\MTSO(3); \bF_p) \to H^1 (B \Diff^+ (M);
\bF_p)$ is also trivial.

Let $\cA_p$ denote the mod $ p$-Steenrod algebra. Then the
restriction of $\alpha$ to $\cA_p \cdot H^1 (\MTSO(3))$ is trivial
because $\alpha$ is a spectrum map. Let $x := p_1 \in H^4 (BSO(3);
\bF_p)$ be the first Pontrjagin class and let $\cP$ be the total
Steenrod power operation. We will compute $\cP (\thom_{-L_3} (x))
\in H^*(\MTSO(3); \bF_p)$. This is done using a formula by Wu
(compare \cite{MS}, Thm 19.7). Let $r=\frac{1}{2}(p-1)$, $\cP^i$ has
degree $4ri$. Wu's formula states, in the present context, that

\[
\cP (\thom_{-L_3} x )= \thom_{-L_3} ((x + x^p)(1+x^r)^{-1}).
\]

Therefore $\cP^i (\thom_{-L_3} x) \in H^{1+4ri}(\MTSO(3);\bF_p)$
agrees with $\thom_{-L_3} (x^{ri+1})$, multiplied by the coefficient
of $z^{ri+1}$ in the power series

\[
(z + z^p)(1+z^r)^{-1}= \sum_{l \geq 0} (-1)^l ( z^{rl+1} +
z^{rl+p}).
\]

It is clear that this coefficient is a unit in $\bF_{p}^{\times}$.
Thus $\cP^i (\thom_{-L_3} x) \in H^{1+4ri}(\MTSO(3);\bF_p)$ is a
generator and we have argued above that $\alpha^* \cP^i
(\thom_{-L_3} x) =0$. This concludes the proof.
\end{proof}

It appears to be quite difficult to find an example of a
$3$-manifold such that $\alpha:\suspinf  B \Diff^+ (M)_+ \to
\MTSO(3) $ is nonzero in cohomology. Many computations which will
not be reproduced here suggest that $\alpha$ could very well vanish
in cohomology with arbitrary coefficients. However:

\begin{prop}
For $M= \bS^3$, the MTW-map $\suspinf BSO(4)_+ \to \MTSO(3)$ of the
universal $\bS^3$-bundle is not nullhomotopic.
\end{prop}

\begin{proof}[Proof (The author owes this argument to O. Randal-Wiliams):]
Recall that $\alpha$ fits into the cofibre sequence \ref{cofibseq}:

\[
\suspinf BSO(4)_+ \stackrel{\alpha}{\to} \MTSO(3)
\stackrel{\eta}{\to} \Sigma \MTSO(4).
\]

If $\alpha$ were nullhomotopic, then there exists a splitting $s:
\Sigma \MTSO(4) \to  \MTSO(3)$ of $\eta$ (i.e. $s \circ \eta =
\id$). In the sequel we assume that the map

\[
\eta^* : H^* (\MTSO(4); \bF_3) \to H^* (\Sigma^{-1}\MTSO(3); \bF_3)
\]

has a right inverse $s^*$ as a map of $\cA_3$-modules and show that
this is impossible. Let $u_{-3} $ be the Thom class of $-L_3$ and we
write $u_{-3} \cdot x$ for $\thom_{-L_3}(x)$ (recall that this is a
module structure); similarly for $\MTSO(4)$. Since $\eta^* u_{-4} =
\Sigma^{-1} u_{3}$, it follows that $s^* \Sigma^{-1} u_{-3} =
u_{-4}$ and thus that

\begin{equation}\label{steenrodcontradiction}
Q u_{-4} = \Sigma^{-1} s^* Q_{u_3}
\end{equation}

for any $Q \in \cA_3$. Put $Q = \cP^3 - \cP^2 \cP^1$. Recall the
formulae

\[
\cP^1 p_1 = p_{1}^{2} +p_2; \;\cP^2 p_1 = p_{1}^{3}; \; \cP^1
p_{1}^{2} = -p_1 (p_{1}^{2} + p_2); \; \cP^1 p_2 = p_1 p_2
\]

for the Steenrod operations on $BSO(4)$ (and hence, by putting $p_2
=0$, also on $BSO(3)$) and the formula

\[
\cP (u_4) = u_{-4} (K(p_1,p_2)),
\]

where $K$ is the multiplicative sequence associated with
$(1+x)^{-1}$. Its lowest terms are $K(p_1,p_2)= 1 - p_1 + p_{1}^{2}
- p_{1}^{3} - p_1 p_2  + \ldots$. From this, we get

\[
(\cP^3 - \cP^2 \cP^1) (u_{-4}) = u_{-4} p_1 p_2
\]

on $\MTSO(4)$. Therefore $(\cP^3 - \cP^2 \cP^1) (u_{-4}) = 0$ (just
put $p_2 =0$ and shift the degrees. This contradicts
\ref{steenrodcontradiction}.
\end{proof}

We conclude this section by asking the question:

\begin{quest}
Does there exist an oriented closed $3$-manifold $M$ and a prime
$p$, such that $\alpha_{E_M}^{*}:\tilde{H}^*(\MTSO(3);\bF_p) \to
\tilde{H}^*(B \Diff^+ (M); \bF_p)$ is nontrivial?
\end{quest}

\section{What happens for manifold bundles with
boundary?}\label{boundarysection}

In this section we study manifold bundles with boundary and ask
whether the vanishing theorem \ref{vanishing} still holds for such
bundles. We have to distinguish two cases. The first case is when we
require the boundary bundle to be trivialized. In this case, Theorem
\ref{vanishing}, interpreted appropriately, is still true. The
second case is when the boundary bundle is not required to be
trivial. In this case the generalized MMM-classes are not defined in
general and therefore the analogue of Theorem \ref{vanishing} does
not make sense, as we will discuss briefly.

\subsection{Manifold bundles with boundary}

Let $M$ be an $n$-dimensional (oriented, smooth, compact) manifold
with boundary. There are two types of smooth $M$-bundles that come
to mind.

We can study the structural group $\Diff^+ (M)$ of all
orientation-preserving diffeomorphisms, with no condition on the
boundary. We say that a bundle with structural group $\Diff^+ (M)$
and fibre $M$ has \emph{free boundary}. Or we can consider the group
$\Diff^+ (M;
\partial)$ of diffeomorphisms of $M$ that coincide with $\id$ on
a small neighborhood of $\partial M$. Bundles with this structural
groups are said to have \emph{fixed boundary}.

\subsection{The Pontrjagin-Thom construction for bundles with boundary}

Let $f:E \to B$ be a manifold bundle with boundary $\partial f
:\partial E \to B$ and of fibre dimension $n$. The isomorphism $T_v
\partial E \oplus \bR \cong T_v E|_{\partial E}$ defines a spectrum map $\eta_E :\bTh (-T_v \partial
E) \to \Sigma \bTh (-T_v E)$ that fits into a commutative diagram
(the rest of the diagram is explained below)

\begin{equation}\label{ptdiagram}
\xymatrix{
 \suspinf B_+ \ar[d]^{\PT_{\partial E}} \ar[dr]^{\stackrel{\PT_{E}}{\simeq
 \ast}}  &  \\
 \bTh (-T_v \partial E) \ar[r]^{ \eta_E}\ar[d]^{\kappa_{\partial E}} & \Sigma \bTh (-T_v E) \ar[d]^{\kappa_E}\\
\MTSO(n-1) \ar[r]^{\eta} & \Sigma \MTSO(n)\ar[d]^{x}\\
   &  \Sigma A
. }
\end{equation}

Choose an embedding $j:E \to B \times [0, \infty ) \times
\bR^{\infty-1}$ such that $\partial E = j^{-1} (B \times \{ 0 \}
\times \bR^{\infty-1})$ and $j(E) \subset B \times [0, 1 ) \times
\bR^{k-1}$. The collapse construction defines a spectrum map

\[
\PT_E:[0, \infty ] \wedge \Sigma^{\infty -1} B_+    \to \bTh (-T_v
E),
\]

here $\infty \in [0,\infty]$ serves as a basepoint. If $ t \gg 1$,
then the composition

\[
\Sigma^{\infty -1} B_+ \cong \{t \} \wedge \Sigma^{\infty -1} B_+
\to [0, \infty ]\wedge \Sigma^{\infty -1} B_+ \stackrel{\PT_E}{\to}
\bTh (-T_v E)
\]

is the constant map. On the other hand, if $t=0$, then

\[
\Sigma^{\infty -1} B_+ \cong \{0 \}_+ \wedge \Sigma^{\infty -1} B_+
\to [0, \infty ] \wedge \Sigma^{\infty -1} B_+
\stackrel{\PT_E}{\to} \bTh (-T_v E)
\]

is the composition $\eta_E \circ \PT_{\partial E}: \Sigma^{\infty-1}
B_+ \to \Sigma^{-1} \bTh (-T_v \partial E) = \bTh (-T_v E|_{\partial
E}) \to \bTh (-T_v E)$. In other words, the Pontrjagin-Thom map
$\PT_E$ can be interpreted as a nullhomotopy of the composition
$\eta_E \circ \PT_{\partial E}$, as displayed in diagram
\ref{ptdiagram}.

Let $A$ be a spectrum, $x: \MTSO(n) \to A$ a map. By composing the
nullhomotopy $\PT_E$ with $x \circ \kappa-E$, we get a nullhomotopy
$P$ of the composition $x \circ \kappa_E \circ \eta_E \circ
\PT_{\partial E}: \Sigma^{\infty-1} B_+ \to \Sigma A$.

Suppose that there is a second nullhomotopy $Q$ of the same map, but
written as the composition

\[
\suspinf B_+ \stackrel{\PT_{\partial}}{\to}  \bTh (-T_v \partial E)
\stackrel{\kappa_{\partial E}}{\to} \MTSO(n-1) \to \Sigma \MTSO(n-1)
\to \Sigma A
\]

(this is the same as $x \circ \eta \circ \alpha_{\partial E}$). Such
a nullhomotopy typically arises from a vanishing theorem concerning
the bundle $\partial E$ and it only involves $\partial E$ and some
choices that do not depend on $E$. We can glue the two
nullhomotopies $Q$ and $P$ and obtain a map $\suspinf B_+ \to A$. If
the nullhomotopy $Q$ is defined for the universal bundle $E_M \to B
\Diff^+ (M)$, we can use this procedure to define characteristic
classes of smooth $M$-bundles. More precisely, even though the map
$\alpha_{E_M}: \suspinf B \Diff^+ (M)_+ \to \MTSO(n)$ does not
exist, we can make sense out of the element $\alpha_{E_M}^{*} (x)
\in A (B \Diff^+ (M))$.

We list a few examples of situations to which the above philosophy
can be applied. In each of these cases, we would need to make the
map $x$ as well as the nullhomotopy $Q$ precise on the point-set
level. We indicate how this works in the example that is of interest
to us: the second example.

\begin{enumerate}
\item If $\partial E = \emptyset$ and $x= \id: \MTSO(n) \to \MTSO(n)$ and $Q$ is the constant nullhomotopy,
then we get of course the MTW-map $\alpha_E$ back.
\item (generalization of the first example) If $x= \id_{\MTSO(n)}$
and $E$ has fixed boundary, then $\alpha_{\partial E}: \suspinf B_+
\to \MTSO(n-1)$ factors as $\suspinf B_+ \to \suspinf \bS^0
\stackrel{\alpha_{\partial M}}{\to} \MTSO(n-1)$. But there is an
oriented nullbordism $W$ of $\partial M$ (of course, $W=M$ is a
possible choice, but there is no reason to prefer this choice). This
nullbordism induces a nullhomotopy of the composition $\suspinf
\bS^0 \stackrel{\alpha_{\partial M}}{\to} \MTSO(n-1) \to \Sigma
\MTSO(n)$. Thus we are in the above situation and hence we can
define a map $\suspinf B_+ \to \MTSO(n)$, called $\alpha_E$.
Geometrically, this corresponds to gluing in the trivial bundle $B
\times W$ into $E$ along $\partial E$. If $\hat{E}$ denotes this new
bundle, then $\alpha_E = \alpha_{\hat{E}}$. This geometric
description, together with the homotopy equivalence $\alpha^{GMTW} :
\Omega B \Cob_n \simeq \loopinf \MTSO(n)$ from \cite{GMTW}, explains
how to the make the nullhomotopy precise. Note that this
construction depends on the choice of $W$; thus $\alpha_E $ is
well-defined only modulo maps of the form $\alpha_{B \times V}$ for
constant bundles of closed manifolds. Therefore the map $\alpha_E :
H^* (\MTSO(n)) \to H^* (B \Diff^* (M))$ does not depend on the
choice of $W$ as long as we consider positive degrees $* > 0$.
\item If $A = \Sigma^k H \bZ$, $y \in H^{k} (BSO(n))$ and $\chi$ is
the Euler class, let $x = \thom (y \chi) \in H^k (\MTSO(n))$. There
is a canonical nullhomotopy of $x \circ \eta$, induced by a nonzero
cross-section of $L_n|_{BSO(n-1)}$. Thus we can apply the above
construction. This shows that, although there is no map
$\alpha_E:\suspinf B_+ \to \MTSO(n)$ for a bundle with boundary, we
can still define what ought to be the pullback $\alpha_{E}^{*} \thom
(y \chi)$, which should be the same as $f_! (y (T_v E) \chi(T_v
E))$. Recall that the Becker-Gottlieb transfer $\trf_f:\suspinf B_+
\to \suspinf E_+$ also exists when $f:E \to B$ is a manifold bundle
with boundary. One can show that $\trf_{f}^{*} (y) = \alpha_{E}^{*}
\thom (y \chi)$.
\item If $n$ is even and $x : \MTSO(n) \to K$ is the universal
symbol class of the signature operator, the vanishing theorem
implies that we can define the index of the signature on an
arbitrary bundle of even dimension. On the other hand, for odd $n$,
it turns out that the index of the signature of the boundary is an
obstruction to define the index of the odd signature operator.
\end{enumerate}

We can use the second example from above to define the MTW-map of
any bundle with fixed boundary. Since, in the above notation,
$\alpha_E$ is defined to be the MTW-map of the closed bundle
$\hat{E}$, Theorem \ref{symbolgoestonull} is still true for the new
MTW-map of a bundle with boundary. All consequences that were
derived from \ref{symbolgoestonull} by formal computations are still
true, namely Theorems \ref{vanishing}, \ref{adamsmethod},
\ref{wumethod}.

\address
\email
\end{document}